\documentclass[dvips,11pt]{article}
\usepackage[english]{babel}
\usepackage{graphicx}
\usepackage{a4wide}
\usepackage{amssymb}
\usepackage{array}
\usepackage{subfigure}
\usepackage{amsfonts}
\usepackage{amsmath}
\usepackage{color}
\newtheorem{thm}{Theorem}[section]
\newtheorem{lem}{Lemma}[section]

\newenvironment{proof}[1][Proof]{\noindent\textbf{#1.}}{\ \rule{0.5em}{0.5em}}

\catcode`\@=11

\@addtoreset{equation}{section}

\catcode`\@=11

\begin{document}
\bibliographystyle{abbrv}
\title{Addendum to ``Travelling waves for a non-local Korteweg-de Vries-Burgers equation'' [J. Differential Equations 257 (2014), no. 3, 720--758]
}
\author{
F. Achleitner\footnote{Vienna University of Technology, 
Institute for Analysis and Scientific Computing, Wiedner Hauptstrasse 8-10, 1040 Wien, Austria.
E-mail: franz.achleitner@tuwien.ac.at
},
C.~M. Cuesta\footnote{University of the Basque Country (UPV/EHU), Faculty of Science and Technology, Department of Mathematics, Aptdo. 644, 48080 Bilbao, Spain. E-mail: carlotamaria.cuesta@ehu.es
}
}

\date{}
\maketitle
\begin{abstract}
We add a theorem to [J. Differential Equations 257 (2014), no. 3, 720--758] by F. Achleitner, C.M. Cuesta and S. Hittmeir. In that paper we studied travelling wave solutions of a Korteweg-de Vries-Burgers type equation with a non-local diffusion term. In particular, the proof of existence and uniqueness of these waves relies on the assumption that the exponentially decaying functions are the only bounded solutions of the linearised equation. In this addendum we prove this assumption and thus close the existence and uniqueness proof of travelling wave solutions.
\end{abstract}

\emph{Keywords.} non-local evolution equation, fractional derivative, travelling waves

\emph{Math.Subj.Class.} 47J35, 26A33, 35C07

\section{Introduction}
In \cite{ACH} we study the existence and stability of travelling waves of the following one-dimensional evolution equation:
\begin{equation}\label{EE}
  \partial_t u + \partial_x u^2 = \partial_x {\cal D}^{\alpha} u +\tau \partial_x^3 u\,, \quad x\in \mathbb{R}\,, \ t\geq 0
\end{equation}
with $\tau >0$, see also \cite{AHS} for the case $\tau=0$. The symbol ${\cal D}^\alpha$ denotes the non-local operator acting on $x$ that, applied to a general function $f:\mathbb{R}\rightarrow\mathbb{R}$, reads
\begin{equation}\label{FD}
{\cal D}^\alpha f(x)=d_\alpha\int_{-\infty}^x \frac{f'(y)}{(x-y)^{\alpha}}dy \,,
\quad \mbox{with} \quad 
0<\alpha<1\,, \quad d_\alpha := \frac{1}{\Gamma(1-\alpha)}>0 \,,
\end{equation}
here $\Gamma$ denotes the Gamma function.

We recall that travelling wave solutions of (\ref{EE}) are solutions of the form $u(x,t)=\phi(\xi)$ with $\xi=x-ct$ and $c\in\mathbb{R}$, that satisfy
\begin{equation}\label{TWP}
h(\phi)= {\cal D}^\alpha \phi+\tau\phi''\,,
\quad \mbox{where} \quad h(\phi) := -c(\phi-\phi_-)+\phi^2-\phi_-^2 \,.
\end{equation}
and
\begin{equation}\label{far-field}
\lim_{\xi\to-\infty}\phi(\xi) = \phi_-\,,\quad \lim_{\xi\to \infty}\phi(\xi) = \phi_+ \,
\end{equation}
(see \cite{ACH} for details) for some constant values $\phi_-$ and $\phi_+$. Here $'$ denotes differentiation with respect to $\xi$. 
Further, it is assumed that $\phi_->\phi_+$ (Lax entropy condition), which implies that $c=\phi_++\phi_-$ (Rankine-Hugoniot wave speed), and also that $h'(\phi_-)=\phi_--\phi_+>0$.

The proof of existence of travelling wave solutions, for both $\tau>0$ and $\tau=0$, relies on the assumption that the functions $v(\xi) = C e^{\lambda\xi}$, $C\in\mathbb{R}$, are the only bounded solutions of the linearised equation
\begin{equation} \label{LTW}
  h'(\phi_-)v={\cal D}^\alpha v + \tau v''\,,
\end{equation}
where the exponent $\lambda>0$ is the real and strictly positive zero of
\begin{equation}\label{pol:left}
P(z)=\tau z^2+z^\alpha-h'(\phi_-)\,.
\end{equation}
We recall that for $\tau>0$ there is a unique positive real zero of (\ref{pol:left}), the other zeros being two complex conjugates with negative real part, see  \cite{ACH}.

In \cite{ACH} we do not give a complete proof of this assumption, however, we prove it in suitable weighted exponential spaces. We show this by writing the equation as a Wiener-Hopf equation (\cite{Wiener+Hopf:1931}) and applying the results by \cite{Krein:1958}. Namely, we show that if 
$0<\mu<\min\{\lambda,h'(\phi_-)/(2-\alpha)\}$, then, all solutions of (\ref{LTW}) that are in the space
 \[
 L^\infty_w(-\infty,0) = \{f \in L^\infty(-\infty,0): \ f(\xi)=e^{\mu \xi}g(\xi) \quad \mbox{for some } \quad g \in L^\infty(-\infty,0)\}
\]
are given by the one-parameter family $\{Ce^{\lambda \xi} :\ C\in \mathbb{R}\}$. A similar result is given in \cite{AHS} for the case $\tau=0$, 
where it is also shown that bounded solutions decay to $0$ as $\xi\to-\infty$ faster than algebraically.

The aim of the current addendum is thus to present an alternative proof that removes the weight of the space. Namely, we show that:
\begin{thm}\label{uniqueness}
All solutions of (\ref{LTW}) with $\tau\geq 0$ that are in $H^s(-\infty,0)$ with $s\geq 2$ are given by the one-parameter family $\{\xi\in(-\infty,0)\rightarrow Ce^{\lambda \xi} :\ C\in \mathbb{R}\}$, where $\lambda$ is the positive zero of (\ref{pol:left}). 
\end{thm}		

The proof is very easy once one realises that the integral
\begin{equation}\label{key}
I[v]:=\int_{-\infty}^0\int_{-\infty}^\xi \frac{v'(\xi)v'(y)}{(\xi-y)^{\alpha}}dyd\xi 
\end{equation}
is non-negative. The proof of this fact is adapted from \cite{LiebLoss}.

Before we give the proof of Theorem~\ref{uniqueness}, let us recall some notation and properties of (\ref{FD}). 
For $s\geq 0$ we shall adopt the following notation for the Sobolev space of square integrable 
functions,
\[
H^s(\mathbb{R}) := \{ u:\,\|u\|_{H^s(\mathbb{R})} <\infty \} \,,\qquad \|u\|_{H^s(\mathbb{R})} := \|(1+|k|^2)^{s/2}\hat{u}\|_{L^2(\mathbb{R})} \,,
\]
and the corresponding homogeneous norm $\|u\|_{\dot H^s(\mathbb{R})} := \| |k|^s\hat{u} \|_{L^2(\mathbb{R})}$.
It is easy to see that $\|{\cal D}^\alpha u\|_{\dot H^s(\mathbb{R})} = \|u\|_{\dot H^{s+\alpha}(\mathbb{R})}$, so that ${\cal D}^\alpha$ is a bounded linear operator from $H^s(\mathbb{R})$ to $H^{s-\alpha}(\mathbb{R})$.


We recall that the analysis in \cite{ACH} starts out by proving a 'local' existence result on $(-\infty,\xi_0]$ with $\xi_0<0$ and $|\xi_0|$ sufficiently
 large. This proof is based on linearisation about $\xi=-\infty$ (or, equivalently, $\phi=\phi_-$), which is given by (\ref{LTW}). It is then assumed that 
\begin{equation} \label{lin-ass} 
\mathcal{N}\big( \tau\partial_\xi^2+ {\cal D}^\alpha- h'(\phi_-) \mbox{Id}\big)
= \mbox{span}\{e^{\lambda \xi}\} \quad \mbox{in}\quad H^s(-\infty,\xi_0)
\end{equation}
where $\mbox{Id}$ denotes the identity operator and $s=4$ if $\tau>0$ and $s=2$ if $\tau=0$. The assumption (\ref{lin-ass}) follows if 
Theorem~\ref{uniqueness} is true. Notice that the problem is invariant under translation, so we can take $\xi_0=0$ without loss of generality 
in order to show (\ref{lin-ass}).

\section{Proof of Theorem~\ref{uniqueness}}
We work in the Hilbert space
\[
H_0^2(-\infty,0) =\{ v\in H^1(-\infty,0): \ v(0)=0\}\cap H^2(-\infty,0)\,.
\]
We need two lemmas. First, we find a way of writing the potential in ${\cal D^\alpha}$ as an integral (see \cite{LiebLoss}):
\begin{lem}\label{h:conv}
Let $\beta>-1$, then there exists a function 
$H\in C_c^\infty(\mathbb{R})$ such that
\[
\int_0^\infty t^{\beta} H(t)dt=1\,.
\]
Then by choosing
\[
H(t)=  \int_{\mathbb{R}} h(t-r)\,h(r)\, dr
\]
for all $t\in \mathbb{R}$ where $h\in C_c^\infty(\mathbb{R})$ is an even function, one can write
\[
|x|^{-(\beta+1)}=\int_0^\infty t^\beta H(tx)dt\,.
\]
Moreover, for any $\xi$, $y\in \mathbb{R}$ and $a>0$, we have
\begin{equation}\label{power}
|\xi-y|^{-a}=\int_0^\infty t^a \int_{\mathbb{R}} h(t(z-\xi)) \, h(t(z-y))\,dz\,dt\,.
\end{equation}
\end{lem}

We can now show the following key result:
\begin{lem}\label{good:sign}
Let $v\in H_0^2(-\infty,0)$, then the integral $I[v]$ in (\ref{key}) is well-defined and is non-negative. 
Moreover, $I[v]$ is zero if and only $v\equiv 0$.
\end{lem}
\begin{proof}
We first observe that $I[v]$ is well-defined. First using the Cauchy-Schwarz inequality and that $\alpha\in(0,1)$, it follows that
\begin{equation}\label{CS}
\left|\int_{-\infty}^0 v'(\xi){\cal D}^\alpha v(\xi) d\xi\right |\leq \|v'\|_{L^2(-\infty,0)}\| D^\alpha v\|_{L^2(-\infty,0)}\,.
\end{equation}
We now use the reflection operator ${\cal E}: H_{0}^2(-\infty,0)\to H^2(\mathbb{R})$, 
\[
{\cal E}[u](x):=u^\ast(x)=\begin{cases} u(x) \quad \mbox{if} \ x\leq 0,\\ -u(-x) \quad \mbox{if} \ x> 0,
\end{cases}
\]
so that $\|u^\ast\|_{L^2(\mathbb{R})}^2=2\|u\|_{L^2(-\infty,0)}^2$. Then,
\[ 
\|{\cal D}^{\alpha}v \|_{L^2(-\infty,0)}^2\leq \|{\cal D}^\alpha v^\ast \|_{L^2(\mathbb{R})}^2=\|v^\ast\|_{\dot H^\alpha(\mathbb{R})}^2\leq \|v^\ast\|_{H^1(\mathbb{R})}^2= 2 \|v\|_{H^1_0(-\infty,0)}^2<\infty\,.
\]
This and ( \ref{CS}) imply that $I[v]$ is well-defined. 

In order to determined the sign of $I[v]$, we first write (\ref{key}) over integrals on $\mathbb{R}$:
\[
I[v]=\int_{-\infty}^0\int_{y}^0  \frac{v'(\xi)v'(y)}{(\xi-y)^{\alpha}}\,d\xi\, dy=\frac{1}{2}\int_{\mathbb{R}}\int_{\mathbb{R}} \frac{\left[v'(\xi)\theta(-\xi)\right]  \left[v'(y)\theta(-y)\right]}{ |\xi-y|^{\alpha}}\,d\xi\, dy
\]
where $\theta$ denotes the Heaviside function. Let us, for simplicity of notation, write $F(x)=v'(x)\theta(-x)$. 
Then, by expressing the potential according to Lemma~\ref{h:conv} (\ref{power}), we obtain 
\[
I[v]=  \frac{1}{2}\int_{\mathbb{R}} \int_{\mathbb{R}} F(\xi)F(y) \int_0^\infty t^{\alpha}\int_{\mathbb{R}} h(t(z-\xi))\,h(t(z-y))\,dz\,dt \,dy\,d\xi
\]
and by Fubini-Tonelli Theorem we have that 
\begin{eqnarray*}
  I[v] 
  & = & \frac{1}{2} \int_0^\infty t^{\alpha} \int_{\mathbb{R}} \left( \int_{\mathbb{R}} F(\xi) h(t(z-\xi))d\xi\right)^2\,dz\,dt\geq 0\,.
\end{eqnarray*}

Now, if $I[v]=0$ then $F\ast h_t=0$ almost everywhere, where $h_t(x)=h(tx)$. 
Since $h$ has compact support, $h_t$ acts as a mollifier as $t\to \infty$ and it is not hard to show that then $F\equiv 0$ (\cite{LiebLoss}). 
Recalling that $F(x)=v'(x)\theta(-x)$ with $v(0)=0$ then $v\equiv 0$.  
\end{proof}

\begin{proof}[Proof of Theorem~\ref{uniqueness}]
First let us prove the uniqueness of solutions of (\ref{LTW}) in $H^2(-\infty,0)$ for a given data in $\xi=0$. 
This is equivalent to proving that the only solution of (\ref{LTW}) in $H^2_0(-\infty,0)$ is $v\equiv 0$; 
indeed, if $v_1$ and $v_2$ are two solutions of (\ref{LTW}) with $v_1(0)=v_2(0)$, then $v=v_1-v_2$ satisfies
\begin{equation}\label{eq:v}
\begin{cases}
 h'(\phi_-) v = {\cal D}^\alpha v +\tau v''\\
v(0)=0\,.
\end{cases}
\end{equation}

Testing (\ref{eq:v}) with $v'\in H^1(-\infty,0)$ and integrating with respect to $\xi$ we obtain:
\[
0= \frac{h'(\phi_-)}{2} v(0)^2= \int_{-\infty}^0 v'(\xi){\cal D}^\alpha v(\xi) d\xi  + \frac{\tau}{2} v'(0)^2 
\] 
and Lemma~\ref{good:sign} implies that $v\equiv 0$.

It is easy to see, just by a straight computation, that the exponential functions $Ce^{\mu\xi}$ with $\mu$ being a zero of (\ref{pol:left}) satisfy (\ref{LTW}). 
 If $\mu=\lambda$, then these exponential functions are the only solutions in $H^2(-\infty,0)$, by the uniqueness just established. On the other hand, 
since these functions are also solutions in $H^s(-\infty,0)$ with $s> 2$ and $H^s(-\infty,0)\subset H^2(-\infty,0)$, the result follows.
\end{proof}

\paragraph{Acknowledgements}
The authors wish to thank Lukas Neumann for helpful comments. 
They also gratefully acknowledge the financial support 
of Vienna University of Technology via its International Office, the financial support of the Spanish MINECO through the project MTM2014-53145-P and the financail support of the Basque Government through the Research Group IT641-13.


\begin{thebibliography}{99}
\bibitem{ACH} F. Achleitner, C.~M. Cuesta, S. Hittmeir.
\newblock Travelling waves for a non-local Korteweg-de Vries-Burgers equation
\newblock{\em J. Differential Equations}, 257: 720--758, 2014.

\bibitem{AHS} F. Achleitner, S. Hittmeir, C. Schmeiser.
\newblock On nonlinear conservation laws with a nonlocal diffusion term. 
\newblock{\em J. Differential Equations}, 250(4): 2177--2196, 2011.


\bibitem{Krein:1958}
M.~G. Kre{\u\i}n. 
\newblock Integral equations on the half-line with a kernel depending on the difference of the arguments.
\newblock{\em Uspehi Mat. Nauk.}, 13(5):3--120, 1958.


\bibitem{LiebLoss} Elliott H.~Lieb and M.~Loss.
\newblock {\em Analysis}. 
\newblock Graduate Studies in Mathematics, 14. American Mathematical Society, Providence, RI, 1997.

\bibitem{Wiener+Hopf:1931}
N. Wiener and E. Hopf.
\newblock 
{\"U}ber eine Klasse singul{\"a}rer Integralgleichungen.  
\newblock{\em Sitz.Ber.Preuss.Akad.Wiss.Berlin}, XXXI: 696-706, 1931.
\end{thebibliography}

\end{document}